\documentclass[final,1p,times,twocolumn]{elsarticle}

\usepackage{amsmath, amssymb, amsthm}

\newtheorem{thm}{Theorem}[section]
\newtheorem{cor}[thm]{Corollary}
\newtheorem{lemma}[thm]{Lemma}
\newtheorem{prop}[thm]{Proposition}
\newtheorem{defn}[thm]{Definition}

\journal{arXiv}

\begin{document}

\begin{frontmatter}



\title{Some automorphism groups of finite $p$-groups}

 \author[label1]{Yassine Guerboussa}
 \author[label2]{Miloud Reguiat}
 \address[label1]{University Kasdi Merbah Ouargla, Ouargla, Algeria \\ {\tt Email: yassine\_guer@hotmail.fr}}
 \address[label2]{University Kasdi Merbah Ouargla, Ouargla, Algeria \\ {\tt Email: Miloud\_Reguiat@hotmail.fr}}



\begin{abstract}
 A conjecture of Berkovich asserts that every non-simple finite $p$-group has a non-inner automorphism of order $p$.  This conjecture is far from being proved despite the great effort devoted to it.  In this peper we prove it for $p$-groups of coclass 2, provided that $p$ is odd.  Some related results are also proved, and may be considered as interesting independently.

\end{abstract}


\end{frontmatter}


\section{Introduction}

Automorphisms of finite $p$-groups are quite mysterious, and many old problems in this area are still open.  May be the most attractive are the divisibility conjecture stating that $|G|$ divides $|Aut(G)|$, if $|G|$ is greater than $p^2$.  And the question or the conjecture of Berkovich stating that every non-simple finite $p$-group has a non-inner automorphism of order $p$.
\\  This conjecture is confirmed for various classes of $p$-groups, regular $p$-groups by Schmid [12], see also Jamali [8] for $p=2$, powerful $p$-groups by Abdollahi [2],  and specially for $p$-groups of class not exceeding $3$ by Liebeck [10], Abdollahi [1] and Abdollahi et al [3 ].  Notably, Deaconescu and Silberberg [4 ] have reduced it to the case when the $p$-group $G$ is  strongly frattinan, that is $C_G(Z(\Phi(G)))=\Phi(G)$.   
\\If $G$ is a $p$-group of order $p^n$ and class $c$, then  the coclass of $G$ is equal to $r=n-c$.  Classifying $p$-groups by considering  the coclass as the primary invariant, has proved a great success.  We refer the reader to  C. Leedham-Green [9], for the details.  At this time, It's not clear how to studying the conjecture of Berkovich in the framework of coclass theory.  And this conjecture is almost studied for groups of small class or some close families, as for example considering $G/Z(G)$ to be powerful as done by A. Abdollahi in [2].
\\It's natural to consider this question for $p$-groups in the opposite side, namely for those with small coclass.  it's easy to see that is true in the maximal class case.  So the next candidate are $p$-groups of almost maximal class, or those of coclass 2.  In section 5, of this paper we prove that the conjecture is true for these $p$-groups provided that $p$ is odd.
\\The reader may not consider the previous as the main result in the paper.  In section 3, and under a suitable condition, we prove a result about the structure of some automorphism groups.  Namely, about the group of automorphisms centralizing $G/H$, where $H$ denotes the full inverse image in $G$ of $\Omega_1(\zeta_2(G)/\zeta_1(G))$. These automorphisms may be considered as the farthest removed from central automorphisms ( see [6]).
\\Before, and in section 2, we have to list some preliminary results which we shall need in the following sections.  The main idea is to consider the additive group $Der(G,A)$ of derivations from $G$ into $A$ as ring, when $A$ is a normal abelian subgroup of $G$, viewed as a $G$-module via conjugation.  And also considering the adjoint group of this ring as a suitable group of automorphisms.  This allows us to capture some of the properties of this automorphism group by means of the properties of the associated ring.
\\A natural sequence of ring homomorphisms 
$$ 0  \rightarrow Der(G,B)  \rightarrow Der(G,A) \rightarrow Der(G/B,A/B)$$ 
arise when $B$ is a normal subgroup of $G$, such that $B \leq A$ and $B^{\delta}\subset B$ for every $ \delta \in Der(G,A)$.  
\\In section 4, we shall caracterize the $p$-groups, $p$ odd, with an exact sequence of the obove type, under the assumption that $A$ is elementary abelian of rank 2. 
\\Throughout this paper, the unexplaned natation is standard in the leterature.  We denote by $\zeta_i(G)$ and $\gamma_i(G)$, the terms of the upper and the lower central series of the group $G$, respectively.  And $d(G)$ will denote the minimal number of generators of $G$. Following A. Mann, we say that a $p$-group $G$ is $p$-central if every element of order $p$ (of order 4, if $p=2$) in $G$ is central.  If $N$ is a normal subgroup of $G$, we denote by $Aut_N(G)$ the group of automorphisms $\sigma$ of $G$ having the property $x^{-1}\sigma(x) \in N$, for all $x \in G$.
\\If $R$ is a ring, then we donote by $R^+$ it's addative group, and by $R^{\circ}$ it's adjoint group.  The $nth$ power of an element $x$ in the group  $R^{\circ}$ is denoted by $x^{(n)}$.
 
\section{\bf Preliminaries}

  Let $R$ be a (associative) ring, not necessarily with an identity element.  The set of all elements of $R$ forms a monoid with identity element $0 \in R$ under the circle composition $x\circ y=x+y+xy$, this monoid is called the adjoint monoid or semi-group of the ring $R$.  The adjoint group $ R^\circ$ of the ring $R$ is the group of invertible elements in the previous monoid.   Following Jacobson, we say that $R$ is radical if his adjoint semi-group is a group, or equivalently $ R^\circ = R$.

The $n th$ power $R^n$ of $R$ is the additive subgroup generated by the set of all products of $n$ elements of $R$.  We say that $R$ is nilpotent if  $R^{n+1}=0$ for some non-negative integer $n$, the least $n$ satisfying $R^{n+1}=0$ is called the degree of nilpotency of the ring $R$.  Notice that every nilpotent ring $R$ is radical since for every $x \in R$ we have$$ x\circ \sum_{i}(-1)^ix^i=(\sum_{i}(-1)^ix^i) \circ x =0.$$
\\The Jacobson radical of the ring $R$ is the largest ideal of $R$ contained in the adjoint group $R^\circ$.  This imply that $R$ is radical if and only if it coincides with its Jacobson radical.  By a classical result the Jacobson radical of an artinian ring is nilpotent, hence every artinian (in particular finite) radical ring is nilpotent.
\begin{lemma}
The adjoint group of a nilpotent ring $R$ is nilpotent of class at most equal to the degree of nilpotency of $R$.
\end{lemma} 

 Following [6],  we say that a ring $R$ is right (left, resp) $p$-nil if every element $x \in R$ satisfying $ p x =0$  is a right (left, resp) annihilator of $R$. The ring $R$ is said to be $p$-nil if it is right and left $p$-nil.
\\For example, for every ring $R$ the subring $S=pR$ is $p$-nil.  Also we check easily that the left and the right annihilators of  $\Omega_1(R^+)$ are respectively right and left $p$-nil.
\\ The following three results are taken from [6].
\begin{lemma}
Let $R$ be a ring with an additive group of finite exponent $p^m$.  If $R$ is  left or right $p$-nil then $R$ is nilpotent of degree at most equal to $m$.  In particular the adjoint group $R^\circ$ is nilpotent of class at most equal to $m$.
\end{lemma}

\begin{lemma}
Let $R$ be a $p$-ring, $p$ odd.  If $R$ is left or right $p$-nil, then $\Omega_{\{n\}}(R^\circ)=\Omega_n(R)$,  for every $n \geq 1$.  In particular we have $\Omega_n(R^\circ)=\Omega_{\{n\}}(R^\circ)$.  
\end{lemma}
Now, let $G$ be a group and $A$ be a normal abelian subgroup of $G$.  Then $A$ can be viewed as a $G$-module via conjugation.  Let denote by $Der(G,A)$ the additive group of derivations from $G$ into $A$, that are the mapping  $\delta:G\rightarrow A$ having the property  $\delta(xy)=\delta(x)^y \delta(y)$, for all $x, y \in G$.  This additive group endowed with the multiplication  $ \delta_1 \delta_2(x)= \delta_2(\delta_1(x))$ has a structure of a ring.
\\ Let $Aut_A(G)$ be the group of the  automorphisms $\sigma$ of $G$ having the property $x^{-1}\sigma(x) \in A$, for all $x \in G$, then we have the following relation.
\begin{prop}
Let $G$ be a group and $A$ be an abelian normal subgroup of $G$.  Then the mapping $\sigma\mapsto \delta_{\sigma}  $, with $\delta_{\sigma}(x)= x^{-1}\sigma(x)$ determines an isomorphism between the group $Aut_A(G)$ and the adjoint group of the ring $Der(G,A)$.
\end{prop}

Let $B$ be a normal subgroup of $G$, such that $B \leq A$ and $B^{\delta}\subset B$ for every $ \delta \in Der(G,A)$.  Then we have a canonical embedding of the ring $Der(G,B)$ in $Der(G,A)$.  
\\Also $A/B$ is a $G/B$-module under the usual action and the map $ \delta\mapsto \tilde{\delta}  $ defined from $Der(G,A)$ to $Der(G/B,A/B)$ by $\tilde{\delta}(xB)=\delta(x)B$ is ring homomorphism.
\begin{lemma}
Under the above notation, we have the following exact sequence of ring homomorphisms $$ 0  \rightarrow Der(G,B)  \rightarrow Der(G,A) \rightarrow Der(G/B,A/B)$$   
\end{lemma}
\begin{proof}
straightforward.
\end{proof}
We also need the following easy observations.
\begin{lemma}
If $G$ is a finite $p$-group, then every maximal subgroup $M$ of $G$ occurs as the kernel of a homomorphism  $r :G \rightarrow \mathbb{Z}_p$.
\end{lemma}
\begin{proof}
 Let be $g \in G-M$.  Then $G=<g>M$ and so every element of $G$ can be written in the form $g^im$, with $m \in M$.  The map $r(g^im)=i\mbox{ mod }p$, defines a homomorphism from $G$ to $ \mathbb{Z}_p$ with a kernel equal to $M$.
\end{proof}
\begin{lemma}
If $G$ is a purely non-abelian finite $p$-group, then  $\Omega_1(Z(G))\leq  \Phi (G)$. 
\end{lemma}
\begin{proof}
Let $z \in  \Omega_1(Z(G))$.  If there exists a maximal subgroup $M$ such that $z \not \in M$, then $G=<z> \times M$. Thus $G$  is not purely non-abelian. 
\end{proof}
\section{\bf {\bf \em{\bf Almost central automorphisms}}}
\vskip 0.4 true cm
Let $H$ be the full inverse image in $G$ of $\Omega_1(\zeta_2(G)/\zeta_1(G))$.  The subgroup $H$ arise naturaly in the study of  central automorphisms of $G$, since the inner central automorphisms of order $p$, are induced exactely by $H$.
\\It's well known that $d(G)d(\zeta(G)) \neq d(H/\zeta_1(G))$ imply that $G$ has a non-inner central automorphism of order $p$.  In fact this condition is also necessary  provided $p>2$ (see [6]).  The $p$-groups of maximal class are straightforward examples satisfying this condition.  Note that $p$-central $p$-groups also satisfy this inequality.  This follows easily from a theorem of Thompson (see [7], III, 12.2, Hilfssatz), and from a theorem of A. Mann [11] fo $p=2$.
\\Here, we prove structural results about the group $Aut_H(G)$, under a suitable condition, namely for finite $p$-groups $G$ satisfying $C_G(\Phi(G)) \leq \Phi(G)$.  It seems that this condition plays the same role in studying  $Aut_H(G)$, as $\zeta_1(G)\leq \Phi(G)$ for central automorphisms (see [6]).   
\begin{thm}
Let $G$ be a finite $p$ group such that $C_G(\Phi(G)) \leq \Phi(G)$.  Let $H$ be  the full inverse image in $G$ of $\Omega_1(inn(G))$ and let denote $H_i = \Omega_i(H)$. Then
\begin{enumerate}
\item the group $Aut_H(G)$ is nilpotent of class at most $min\{r,s\}+1$.  Where $p^r=exp(G/\gamma_2(G))$ and $ p^s=exp(\zeta(G))$.
\item   $Aut_{H_1}(G)$ is nilpotent of class at most $2$.  
\item If $p>2$, then $Aut_H(G)$ has a regular power structure in the sens $$ \Omega_i(Aut_H(G))= \Omega_{\{i\}}(Aut_H(G))=Aut_{H_i}(G).$$
\end{enumerate}
\end{thm}
Before proving this theorem, we need the following lemmas.
\begin{lemma}
 The subgroup $H$ lies in $C_G(\Phi(G))$.  Moreover, under the condition  $C_G(\Phi(G)) \leq \Phi(G)$, the subgroup $H$ is abelian.
\end{lemma}
\begin{proof}
Every element $h \in H$ defines a homomorphism $x \mapsto [x,h]$ from $G$ to its center.  Since $[x,h]^p=[x,h^p]=1$, the image of the previous homomorphism is elementary abelian.  Thus $G/C_G(h)$ is elementary abelian, and so $\Phi(G) \leq   C_G(h)$.  Therefore $h \in C_G(\Phi(G))$.  If $C_G(\Phi(G)) \leq \Phi(G)$, then $H \leq Z( \Phi(G))$.  Hence $H$ is abelian.  
\end{proof} 
In the remaining part of this section $D$ and $D_1$ denote respectively the rings $Der(G,H)$ and  $Der(G,H_1)$.

 \begin{lemma} 
Under the condition $C_G(\Phi(G)) \leq \Phi(G)$, we have
\begin{enumerate}
\item $D^2 \subset Hom(G,\zeta(G))$.
\item ${D_1}^3=0$. 
\item The factor ring $D/D_1$ is right $p$-nil, and has exponent at most $p^{min\{r,s\}}$.  Where $p^r=exp(G/\gamma_2(G))$ and $ p^s=exp(\zeta(G))$. 
\end{enumerate}
\end{lemma}
\begin{proof}
\begin{enumerate}
\item  Let take $A=H$ and $B=\zeta(G)$, in lemma 2.5.  Since $H/\zeta(G)$ lies in the center of $G/\zeta(G)$, we have $Der(G/\zeta(G),H/\zeta(G)) = Hom(G/\zeta(G),H/\zeta(G)) $.  And since $H/\zeta(G)$ is elementary abelian, $H/\zeta(G) \leq \Phi(G)/\zeta(G)$ is contained in the kernel of every homomorphism in $ Hom(G/\zeta(G),H/\zeta(G)) $.  This imply that $ Hom(G/\zeta(G),H/\zeta(G)) $ is a null ring.  By lemma 2.5, the ring $D/Hom(G,\zeta(G))$ can be embedded in $ Hom(G/\zeta(G),H/\zeta(G)) $.  And so  $D^2 \subset Hom(G,\zeta(G))$.   
\item  Let take $A=H_1$ and $B=\Omega_1(\zeta(G))=Z_1$.  We have $H_1/Z_1$ is central in $G/Z_1$ and has exponent $p$.  A similar argument to that in (1) imply that ${D_1}^2 \subset Hom(G,Z_1)$.  Since every homomorphism in $Hom(G,Z_1)$ contains $\Phi(G)$ in its kernel, $Hom(G,Z_1)$ is a right annihilator of $D_1$.  Thus ${D_1}^3=0$.   
\item Let $\delta +D_1 \in \Omega_1(D/D_1)$.  Then $p^2 \delta=0$.  And since $p\delta(x)=\delta(x)^p \in \zeta(G) \mbox{, for all } x \in G$, we have $p\delta \in Hom(G,Z_1)$.  Thus $p \delta'\delta=0 \mbox{ for all } \delta' \in D$, since $Hom(G,Z_1)$ is a right annihilator of the ring $D$. This imply that $ \delta'\delta \in D_1$, and so the ring $D/D_1$ is right $p$-nil.  On the other hand, the homomorphism of group $ D \rightarrow Hom(G,\zeta(G))$ which maps $\delta$ to $p\delta$, induces a monomorphism  $ D/D_1 \rightarrow Hom(G,\zeta(G))$.  Thus $exp(D/D_1) \leq exp(Hom(G,\zeta(G)))= p^{ min\{r,s\}}$.  The result follows.  
\end{enumerate}
\end{proof} 
\begin{proof}[Proof of theorem 3.1]
\begin{enumerate}
\item  Since  $D/D_1$ is a right $p$-nil ring, by lemma 2.2, $D^{min\{r,s\}+1} \subset D_1$.  And by lemma 3.3, 1., we have $D^{min\{r,s\}+1} \subset Hom(G,\zeta(G))$.  It follows that $D^{min\{r,s\}+1} \subset Hom(G,Z_1)$.  Since $Hom(G,Z_1)$ is a right annihilator of the ring $D$, we have $D^{min\{r,s\}+2}=0$.  Hence $D$ has nilpotency degree at most $min\{r,s\}+1$, and so   $Aut_H(G) \cong D^{\circ}$ has nilpotency class at most $min\{r,s\}+1$.     
\item  Since $Aut_{H_1}(G) \cong {D_1}^{\circ}$, it follows immediately from lemma 3.3, 2. that $Aut_{H_1}(G)$ has nilpotency class at most $2$. 
\item   We begin by showing that $D/\Omega_n(D^+)$ is a right p-nil ring, for $n \geq 1$.  If $\delta +\Omega_n(D^+) \in \Omega_1(D/\Omega_n(D^+))$, then $p\delta \in \Omega_n(D^+)$.  Hence   $p^n\delta \in D_1 \cap  Hom(G,\zeta(G))$, and so  $p^n\delta  \in  Hom(G,Z_1)$.  It follows  that $p^n \delta'\delta=0 \mbox{ ,for all } \delta' \in D$.  Thus $\delta'\delta \in  \Omega_n(D^+)$.  
\\ Now, we claim that $\Omega_{\{n\}}(D^\circ)=\Omega_n(D^+)$,  for every $n \geq 1$.  For $n=1$ we have $\delta \in  \Omega_1(D^+)$, or equivalently  $\delta \in D_1$, imply that  $\delta^3=0$.  Thus $$ \delta^{(p)}= \sum_{i\geq 1} \binom{p}{i}\delta^i=p\delta+\binom{p}{2}\delta^2 =0$$
hence  $\delta \in \Omega_{\{1\}}(D^\circ)$.  Conversely, if  $\delta^{(p)}= 0$, then in particular $\delta^{(p)}= 0 \mbox{ mod }D_1$.  Since $D/D_1$ is right $p$-nil, by lemma 2.3 we have  $p\delta =0 \mbox{ mod }D_1$.  Hence $\delta^2 \in D_1$, and by lemma 3.3, 1. $\delta^2 \in Hom(G,\zeta(G))$.  It follows that  $\delta^2 \in Hom(G,Z_1)$, which imply $\delta^3=0$.  Thus $ 0=\delta^{(p)}=p\delta+\binom{p}{2}\delta^2 =p\delta$, and so $\delta \in  \Omega_1(D^+)$.
\\Now we proceed by induction on $n$.  We have $\delta \in  \Omega_n(D^+)$ imply  $\delta +\Omega_{n-1}(D^+) \in \Omega_1(D/\Omega_{n-1}(D^+))$.  Since $D/\Omega_{n-1}(D^+)$ is a right p-nil ring, $(\delta +\Omega_{n-1}(D^+))^{(p)}=  \Omega_{n-1}(D^+)$.  Hence $ \delta^{(p)} \in \Omega_{n-1}(D^+)$, and by induction $ \delta^{(p)} \in \Omega_{\{n-1\}}(D^\circ)$.  Thus $ \delta^{(p^n)}=0$.  This shows that  $\Omega_n(D^+) \subset \Omega_{\{n\}}(D^\circ)$.  The inverse inclusion follows similarly.
\\Finally, since $\Omega_n(D^+)$ is an ideal in the ring $D$, it follows that $\Omega_{\{n\}}(D^\circ)$ is a subgroup of $D^\circ$.  Thus $\Omega_n(D^\circ)=\Omega_{\{n\}}(D^\circ)$.  And since $ \Omega_n(D^+)= Der(G,H_n)^{\circ} =Aut_{H_n}(G)$, the result follows. 
\end{enumerate}

\end{proof}

\section{\bf {\bf \em{\bf $p$-groups with an exact sequence}}}

Throughout this section, unless otherwise stated, $G$ denotes a finite $p$-group, where $p$ is an odd prime.  And $A$ denotes a normal subgroup of $G$, which is elementary abelian of rank $2$, and not central.
\\we wish to caracterize the class of the groups $G$ such that the following sequence is exact.
 $$ 0  \rightarrow Der(G,Z_1)  \rightarrow Der(G,Z_1) \rightarrow Der(G/Z_1,A/Z_1)\rightarrow 0$$ 
Where $Z_1$ denotes the intersection of $A$ with the center.
\\Notice that the above sequence can be written as 
 $$ 0  \rightarrow Hom(G,Z_1)  \rightarrow Der(G,Z_1) \rightarrow Hom(G/Z_1,A/Z_1)\rightarrow 0$$
and by lemma 2.5 we have only to focus on its exacteness in $ Hom(G/Z_1,A/Z_1)$
\\To work in a general framework, it's convenient to consider the following notion.
\begin{defn}
Let $C$ be a maximal subgroup of $G$.  We say that $G$ is full with respect to $C$ if every maximal subgroup $M \neq C$ contains a subgroup $K$ such that : 
\begin{enumerate}
     \item  K is not normal and has index $p^2$ in $G$
     \item   $K \cap  C$ is a normal subgroup of $G$ which contains the subgroup $G^p$. 
\end{enumerate}

\end{defn}
Let us note some easy facts related to the above definition.

\begin{enumerate}
     \item  Clearly, if $G$ is full with respect to one of its maximal subgroups, then $G$ can not be powerful.
     \item   In order for $G$ to be full with respect to a maximal subgroup $C$ it's necessary and sufficient that $G/\gamma_3(G)G^p$ is full with respect to $C/\gamma_3(G)G^p$.
     \item  If $G$ is full with respect to $C$, then $G$ is full with respect to $C^{\sigma}$ for every $\sigma \in Aut(G)$.
\end{enumerate}

 The class of p-groups defined above is very large as shows the following proposition.
 \begin{prop}
If $G$ is $2$-generated and not powerful, then $G$ is full with respect to all its maximal subgroups.
\end{prop}
\begin{proof}
Let  $C$ be a maximal subgroup of $G$. 
\\Let $N=\gamma_3(G)G^p$, and let denote by $G_1$ the section  $G/N$.  It follows immediately that $G_1$ has exponent $p$ and class at most $2$,  and so $\Phi (G_1) \leq Z(G_1)$, that is  $\Phi (G)/N \leq Z(G/N)$.\\
  If $G_1$ is abelian, then $\gamma_2(G) \leq \gamma_3(G) G^p$.  This imply that $\gamma_2(G) \leq  G^p$, which means that $G$ is powerful, a contradiction.  Thus $G$ has class exactely $2$.
\\Now let $M \neq C$  be a maximal subgroup of $G$.  If $M_1=M/N$ is cyclic, then $|M_1| = p$, since $G_1$ has exponent $p$.  This imlpy that $|G_1|=p^2$, thus $G_1$ is abelian, a contradiction.  Therefore $M_1$ is not cyclic, and so it contains at least $p+1$ maximal subgroup.  Let $K/N$ be maximal subgroup in $M/N$ which is distinct from  $\Phi (G)/N$.
\begin{description}
     \item[i)] It follows immediately that $K \leq M$ has index $p^2$.  And if $K$ is normal in $G$, then  $\gamma_2(G) \leq K$.  And since $G^p \leq K$ we obtain  $\Phi (G) = K$, a contradiction.  Thus $K$ is not normal in $G$.  
     \item[ii)]  We have ${K \cap  C} \leq {M \cap C} = \Phi (G)$, hence ${(K \cap  C) /N} \leq  {\Phi (G)/N} \leq Z(G/N)$.  Thus $K \cap  C$ is normal in $G$.  
\end{description}
The result follows. 
  \end{proof}
Now we state the main theorem of this section.

\begin{thm}
Assume that $G$ is purely non-abelian, $C$ is the centralizer of $A$ in $G$, and $Z_1=A \cap Z(G)$.  Then the sequence
 $$ 0  \rightarrow Hom(G,Z_1)  \rightarrow Der(G,Z_1) \rightarrow Hom(G/Z_1,A/Z_1)\rightarrow 0$$
is exact if and only if $G$ is full with respect to $C$. 

\end{thm}
We need the following three lemmas before embarking on the proof.  Throughout these,  $G$ is assumed to be full with respect to $C=C_G(A)$.  And $K$ will denote a subgroup of a fixed maximal subgroup $M \neq C$, which satisfy the conditions 1. and 2. of  definition 4.1.
\begin{lemma}
Let $y \in C- M$  and $x \in M- K$.  Then the map $k \mapsto \alpha (k)$ defined from $K$ into $\mathbb{Z}_p$ by the relation 
$$ [k,y]=x^{\alpha (k)}  \mbox { mod } K$$
is a group homomorphism, whose kernel is equal to $K \cap C$.
\end{lemma}

\begin{proof}
Clearly, $K \neq K \cap C$ and $K/ K \cap C \cong KC/C$ has order $p$.  It follows that $K \cap C$ has index $p^3$ in $G$, and so  $\gamma_3(G) \leq K \cap C$.
\\ Let $k, k' \in K$.  Then
$$ [kk',y] = [k,y] [k,y,k'][k',y]$$ 
hence
$$ [kk',y] = [k,y] [k',y] \mbox { mod } K$$
that is $$x^{\alpha (kk')} = x^{\alpha (k)}x^{\alpha (k')} \mbox { mod } K.$$
This shows that $\alpha$ is a homomorphism.
\\If $\alpha$ is null, then $<y> \leq N_G(K)$.  And since $K$ is normalized by $M$, this imply that $K \lhd G$, a contradiction.  We claim that $K \cap C \leq ker(\alpha)$, which imply that  $K \cap C = ker(\alpha)$.
This follows immediately from the fact $K \cap C$ is normal in $G$. 
  
\end{proof}
\begin{lemma}
For every $u \in A- Z_1$, there exists $z \in Z_1$ such that $ [k,u]=z^{\alpha (k)}  \mbox { for all }k \in K$.  Where $\alpha : K \rightarrow \mathbb{Z}_p $ is the homomorphism defined in lemma 3.4, and $Z_1=A \cap Z(G)$. 

\end{lemma}
\begin{proof}
Since $|A|=p^2$, we have $A \leq \zeta_2(G)$.  Therefore  $ [k,u]\in A\cap Z(G) =Z_1 \mbox { for all }k \in K$.  It follows that the relation $k \mapsto [k,u]$ is a homomorphism from $K$ into $Z_1$, and whose kernel is $K \cap C$.
\\ If we fix a non-trivial element $z_0$ of $Z_1$, we obtain a homomorphism $ \beta : K \rightarrow \mathbb{Z}_p$ with $ [k,u]={z_0}^{\beta (k)}$. 
\\Now we have two homomorphisms $\alpha$ and $\beta$ defined from $K$ onto $\mathbb{Z}_p$ and whose kernels are equal to  $K \cap C$.  We can identify them to their induced homomorphisms in $Hom(K/(K \cap C),\mathbb{Z}_p)$, which is isomorphic to $\mathbb{Z}_p$.  Hence there exists a positive integer $i$ such that $ \beta= i \alpha$, and so we have for $z=z_0^i $, $$[k,u]=z_0^{\beta (k)}=z^{\alpha (k)} \mbox { for all }k \in K$$ 
The result follows.  
\end{proof}
The following lemma may be considered as the key to prove theorem 4.3.
\begin{lemma}
Let be $y \in C- M$  and $x \in \Phi(G)- K$.  Then for every $u \in C- Z_1$, there exists $z \in Z_1$ such that $\delta: G \rightarrow A$ defined by $\delta(kx^jy^i)=z^ju^i$ is a derivation.  Where $k \in K$ and $i,j$ are non-negative integers.

\end{lemma}
\begin{proof}
Let $z$ be as defined in lemma 4.5.

  First we show that $\delta$ is a well defined map.  Clearly, every element $g \in G$ can be written as $g = kx^jy^i$, where  $ k \in K$  and $ i,j \in \mathbb{N} $.  If  $g = kx^jy^i = k'x^{j'}y^{i'} $, then $y^i=y^{i'} \mbox { mod } M$.  Thus $ p$ divides $ i-i'$, that is $ i-i'=pl$ for some integer $l$.  We have $ kx^jy^{pl} = k'x^{j'} $.  Since $ y^{pl} \in G^p \leq K$, we have  $ x^j = x^{j'} \mbox { mod } K $, thus $p$ divides $j-j'$.  Since $A$ is elementary abelian, we have $u^{i-i'}=z^{j-j'}=1$, and so $z^ju^i=z^{j'}u^{i'}$.  
 \\ Now we prove that $\delta$ is a derivation.  Let $g= kx^jy^i$ and $g'= k'x^{j'}y^{i'} $.  We have
\begin{eqnarray}
gg'&=& kx^j[y^{-i},x^{-j'} k'^{-1}]k'x^{j'}y^{i+i'}\nonumber\\ 
 &=&  kx^j [y^{-i}, k'^{-1}][y^{-i},x^{-j'}]^{k'^{-1}}k'x^{j'}y^{i+i'}\nonumber
\end{eqnarray}
since $K \cap C$ is maximal in $K$, $K \cap C$ has index $p^2$ in $C$.  And since $K \cap C$ is normal in $G$, we have  $\gamma_2(C) \leq K \cap C $. 
\\ Thus $k_1 = [y^{-i},x^{-j'}]^{k'^{-1}} \in   K \cap C$.
\\Also, we have 
\begin{eqnarray}
 [y^{-i}, k'^{-1}] &=& \prod_{r=i-1}^{0} [y^{-1}, k'^{-1}]^{y^{-r}}\nonumber\\ 
                            &=&  \prod_{r=i-1}^{0} [y^{-1}, k'^{-1}] [y^{-1}, k'^{-1},{y^{-r}}]\nonumber\\ 
                            &=& [y^{-1}, k'^{-1}]^i k_2 \nonumber  
\end{eqnarray}
  for some $k_2 \in K$
\\And we have   
\begin{eqnarray}
[y^{-1}, k'^{-1}]^i &=& ([y, k'^{-1}]^{y^{-1}})^{-i}= ([k'^{-1},y]^i)^{y^{-1}}\nonumber\\ 
 &=& [k'^{-1},y]^i [[k'^{-1},y]^i,y^{-1}] \nonumber\\  
&=& x^{-i  \alpha (k')}k_3 \nonumber
\end{eqnarray}
 for some $k_3 \in K$.  It follows that
\begin{eqnarray}
gg'&=& kx^j  x^{- \alpha (k')}k_3 k_2 k_1 k'x^{j'}y^{i+i'}\nonumber\\
&=&  k_4 x^{j+j'- i\alpha (k')}y^{i+i'}  \nonumber
\end{eqnarray}
for some $k_4 \in K$.  And thus $ \delta (gg')= z^{j+j'- i \alpha (k')}u^{i+i'} $
\\ On the other hand,
$$ \delta (g)^{g'} \delta (g')=(z^ju^i)^{ k'x^{j'}y^{i'}}z^{j'}u^{i'}$$
using the fact that $x^{j'}y^{i'} \in C_G(A) = C$, we obtain

\begin{eqnarray}
\delta (g)^{g'} \delta (g') &=& (z^ju^i)^{ k'}z^{j'}u^{i'} \nonumber\\
&=& z^j(u^{ k'})^iz^{j'}u^{i'} \nonumber\\
&=&  z^ju^i [u, k']^i z^{j'}u^{i'}   \nonumber
\end{eqnarray}
by lemma 3.5, $[u,k']^i= z^{-i  \alpha (k')}$, hence 
$$ \delta (g)^{g'} \delta (g')= z^{j+j'- i \alpha (k')}u^{i+i'}.$$
The result follows.

\end{proof}
\begin{proof}[Proof of theorem 4.3]
First we have to show that if $G$ is full with respect to $C=C_G(A)$, then the  following sequence is exact.  
 $$  Der(G,A) \rightarrow Hom(G/Z_1,A/Z_1)\rightarrow 0$$ 
  Let $f \in Hom(G/Z_1,A/Z_1)$.  If $f=0$, then it can be lifted to the trivial derivation in $ Der(G,A)$.  Hence we shall assume that $f \neq 0$.  Since $|A/Z_1|=p$, the kernel of $f$ is a maximal subgroup in $G/Z_1$, and so have the form $ M/Z_1$, where $M$ is a maximal subgroup in $G$.  we have to examine the two cases:
\begin{enumerate}  
\item $M \neq C$
\\Let $y \in C-M$, and let $u \in A-Z_1$ such that $ uZ_1 = f(yZ_1)$.  Let $K$ be a subgroup of $M$ satisfying the conditions (1) and (2) of definition 4.1, and let $ x \in \phi(G)-K$.  By lemma 4.6,  there exists $z \in Z_1$ such that $\delta: G \rightarrow A$ defined by $\delta(kx^jy^i)=z^ju^i$ is a derivation, and clearly $f$ is induced by this derivation. 

\item $M=C$. 
Let $t \in G-C$, and let $u \in A-Z_1$ such that $ uZ_1 = f(tZ_1)$.  It follows from the identity $$ (xu)^p=x^pu^p[x,u]^{\binom{p}{2}} \mbox{ for all } x \in G$$
 that $$u^{1+t+...+t^{p-1}}=1$$
  It's straightforward now, to check that $$\delta ( t^im)= u^{1+t+...+t^{i-1}}$$
is derivation from $G$ into $A$, which induces $f$.
\end{enumerate}
Conversely, assume that the preceeding sequence is exact.  Let $M \neq C$ be a maximal subgroup in $G$.  By lemma 2.7, $M/Z_1$ is maximal in $G/Z_1$.  And by lemma 2.6, $M/Z_1$ is the kernel of a homomorphism $r :G/Z_1 \rightarrow \mathbb{Z}_p$.  If $u \in A-Z_1$, we can define a homomorphism $f :G/Z_1 \rightarrow A/Z_1$ with $$f(xZ_1)=(uZ_1)^{r(xZ_1)}$$ 
This homomorphism has kernel  $M/Z_1$, and can be lifted, by our assumption, to a derivation $\delta: G \rightarrow A$.
Consider $K= ker \delta = \{ x\in G, \delta(x)=1\}$.
\\We have $K \leq M$, otherwise $G=MK$ and thus $G/Z_1 \leq ker(f)$, a contradiction.  Thus $K$ coincides with the kernel of the restriction  $\delta/M: M \rightarrow A$.  Clearly, $\delta$ maps $M$ to $Z_1$, and so  $\delta/M: M \rightarrow Z_1$ is a homomorphism of group.  This imply that $K$ has index at most $p^2$ in $G$.
\\ If $K$ is normal in $G$ then $\gamma_2(G) \leq K$.  Let $y \in C-M$ such that $\delta(y)=u$,with $ u \in A-Z_1$.  Then for every $m\in M$ we have $$ \delta (my)= \delta (m)^{y} \delta (y)= \delta (m)u.$$  On the other hand we have  $$ \delta (my)= \delta (ym[m,y])= \delta (ym)^{[m,y]} \delta ([m,y])= \delta (ym)=u^m \delta (m) $$ 
This imply that $u^m=u$, for all $m\in M$, hence $M \leq C_G(u)=C$, a contradiction.  Thus $K$ is not normal in $G$, and has index exactely $p^2$ in $G$.
\\It's straightforward to chek that $\delta(k^g)=1$, for all $k\in K\cap C$ and $g \in G$.  This imply that $ K\cap C \lhd G$.  And finally we have $G^p \leq K$, since for every $x\in G$ we have
\begin{eqnarray}
 \delta(x^p) &=& \prod_{r=o}^{p-1} \delta(x)^{x^i} =  \prod_{r=o}^{p-1} \delta(x) [\delta(x),x^i]  \nonumber\\
                     &=& \delta(x)^p [\delta(x),x]^{\binom{p}{2}}=1. \nonumber                      
  \end{eqnarray}
The theorem follows.
   \end{proof}
\begin{cor}
Let $G$ be a $p$-group which is full with respect to all its maximal subgroups and has cyclic center.  Then for every elementary abelian normal subgroup $A$ of $G$ such that $A \leq \zeta_2(G)$, the sequence
 $$ 0  \rightarrow Hom(G,Z_1)  \rightarrow Der(G,A) \rightarrow Hom(G/Z_1,A/Z_1)\rightarrow 0$$
is exact. Where $Z_1 = A \cap \zeta(G)$. 
\end{cor}
\begin{proof}
Let be $f \in  Hom(G/Z_1,A/Z_1)$, and $(\bar{v}_1,...,\bar{v}_s)$ a basis for the $\mathbb {Z}_p$-vector space $A/Z_1$.  For each $i \in \overline {1,s}$; let $A_i$ be the subgroup of $A$ defined by $ A_i/Z_1 \cong <\bar{v}_i>$.  It follows imediately that $A_i$ is an elementary abelian normal subgroup of $G$ of rank $2$, and 
$$ Hom(G/Z_1,A/Z_1) \cong \bigoplus_{i}Hom(G/Z_1,A_i/Z_1)$$
Thus $f$ can be written as $f=\sum_{i} f_i$, where $f_i \in Hom(G/Z_1,A_i/Z_1)$.  The above theorem imply that each $f_i$ can be lifted to a derivation $\delta_i : G \rightarrow A_i \subset A$.  It follows that $\delta = \sum_{i} \delta_i :  G \rightarrow A$ is a derivation, which induces $f$. 

\end{proof}
\section{\bf {\bf \em{\bf Non-inner automorphisms of $p$-groups of coclass 2}}}
\vskip 0.4 true cm
The results developped in the previous sections allow us to prove the conjecture of Berkovich for $p$-groups of almost maximal class.
\begin{thm}
If $G$ is a finite $p$-group of coclass 2, $p$ odd, then $G$ has a non-inner automorphism of order $p$.
\end{thm}
\begin{proof}
Let   $G$ be a finite $p$-group of coclass 2.  Clearaly, we have to assume that $d(G)d(\zeta(G)) = d(\zeta_2(G)/\zeta_1(G))$.  Since $G$ has coclass $2$, $\zeta_2(G)/\zeta_1(G)$ has at most order $p^2$.  So $ d(\zeta_2(G)/\zeta_1(G)) \leq 2$.  This combined with the above equality imply that $|\zeta_2(G)/\zeta_1(G)|=p^2$, $|\zeta_1(G)|=p$ and $d(G)=2$.  So $G$ is $2$-generated and has cyclic center, and by the main result in [4], we have to assume that $C_G(Z(\Phi(G)))=\Phi(G)$, and so $C_G(\Phi(G)) \leq \Phi(G)$.  
\\Let $H$ be  the full inverse image in $G$ of $\Omega_1(inn(G))$ and let denote $H_1 = \Omega_1(H)$.  By corollary 4.7, we have the exact sequence  
 $$ 0  \rightarrow Hom(G,Z_1)  \rightarrow Der(G,H_1) \rightarrow Hom(G/Z_1,H_1/Z_1)\rightarrow 0.$$
This imply that
$$  |Der(G,H_1)|=|Hom(G,Z_1)||Hom(G/Z_1,H_1/Z_1)|$$
We have 
$$|Hom(G,Z_1)|=|Hom(G/\Phi(G),Z_1)|=|Hom(\mathbb{Z}_p \oplus \mathbb{Z}_p,\mathbb{Z}_p)|=p^2 $$
Similarly, we have
  $$|Hom(G/Z_1,H_1/Z_1)| = p^2$$
Therefore, $  |Der(G,H_1)|=p^4$.  Thus $Aut_{H_1}(G)=p^4$.
\\Now, let denote by $\tau_g$ the inner automorphism induced by $g \in G$. If $\tau_g$ lies in $Aut_{H_1}(G)$, then $[x,g] \in H_1 \leq \zeta_2(G)$.  Therefore $g \in \zeta_3(G)$.  Since $\zeta_3(G)$ has order $p^4$, the group of inner automorphisms induced by $\zeta_3(G)$ has order $|\zeta_3(G)/\zeta(G)|=p^3$.  Thus $Aut_{H_1}(G)$ contains non-inner automorphisms.  The theorem follows now from theorem 3.1, 3.   

\end{proof}

\vskip 0.4 true cm

\begin{center}{\textbf{Acknowledgments}}
\end{center}
This work is done under the supervision of Professor Bounabi Daoud, We are grateful to him for many valuable comments.  The first author is indebted to Professor C. Leedham-Green for his stimulating suggestions.
\vskip 0.4 true cm


References


\bibliographystyle{model1a-num-names}
\bibliography{<your-bib-database>}



\end{document}